\documentclass[10pt]{amsart}
\usepackage{amsmath,amssymb}
\usepackage{verbatim}
\usepackage{color}
\usepackage{hyperref}

\usepackage[top=1.in, bottom=1.1in, left=1.1in, right=1.1in]{geometry}
\usepackage[utf8]{inputenc}

\begin{document}

\newtheorem{thm}{Theorem}[section]
\newtheorem{theorem}{Theorem}[section]
\newtheorem{lem}[thm]{Lemma}
\newtheorem{lemma}[thm]{Lemma}
\newtheorem{prop}[thm]{Proposition}
\newtheorem{proposition}[thm]{Proposition}
\newtheorem{cor}[thm]{Corollary}
\newtheorem{defn}[thm]{Definition}
\newtheorem{remark}[thm]{Remark}
\newtheorem{conj}[thm]{Conjecture}

\numberwithin{equation}{section}

\newcommand{\Z}{{\mathbb Z}} 
\newcommand{\Q}{{\mathbb Q}}
\newcommand{\R}{{\mathbb R}}
\newcommand{\C}{{\mathbb C}}
\newcommand{\N}{{\mathbb N}}
\newcommand{\FF}{{\mathbb F}}
\newcommand{\fq}{\mathbb{F}_q}
\newcommand{\rmk}[1]{\footnote{{\bf Comment:} #1}}

\newcommand{\bfA}{{\boldsymbol{A}}}
\newcommand{\bfY}{{\boldsymbol{Y}}}
\newcommand{\bfX}{{\boldsymbol{X}}}
\newcommand{\bfZ}{{\boldsymbol{Z}}}
\newcommand{\bfa}{{\boldsymbol{a}}}
\newcommand{\bfy}{{\boldsymbol{y}}}
\newcommand{\bfx}{{\boldsymbol{x}}}
\newcommand{\bfz}{{\boldsymbol{z}}}
\newcommand{\F}{\mathcal{F}}
\newcommand{\Gal}{\mathrm{Gal}}
\newcommand{\Fr}{\mathrm{Fr}}
\newcommand{\Hom}{\mathrm{Hom}}
\newcommand{\GL}{\mathrm{GL}}

\renewcommand{\mod}{\;\operatorname{mod}}
\newcommand{\ord}{\operatorname{ord}}
\newcommand{\TT}{\mathbb{T}}
\renewcommand{\i}{{\mathrm{i}}}
\renewcommand{\d}{{\mathrm{d}}}
\renewcommand{\^}{\widehat}
\newcommand{\HH}{\mathbb H}
\newcommand{\Vol}{\operatorname{vol}}
\newcommand{\area}{\operatorname{area}}
\newcommand{\tr}{\operatorname{tr}}
\newcommand{\norm}{\mathcal N} 
\newcommand{\intinf}{\int_{-\infty}^\infty}
\newcommand{\ave}[1]{\left\langle#1\right\rangle} 
\newcommand{\Var}{\operatorname{Var}}
\newcommand{\Prob}{\operatorname{Prob}}
\newcommand{\sym}{\operatorname{Sym}}
\newcommand{\disc}{\operatorname{disc}}
\newcommand{\CA}{{\mathcal C}_A}
\newcommand{\cond}{\operatorname{cond}} 
\newcommand{\lcm}{\operatorname{lcm}}
\newcommand{\Kl}{\operatorname{Kl}} 
\newcommand{\leg}[2]{\left( \frac{#1}{#2} \right)}  
\newcommand{\Li}{\operatorname{Li}}

\newcommand{\sumstar}{\sideset \and^{*} \to \sum}

\newcommand{\LL}{\mathcal L} 
\newcommand{\sumf}{\sum^\flat}
\newcommand{\Hgev}{\mathcal H_{2g+2,q}}
\newcommand{\USp}{\operatorname{USp}}
\newcommand{\conv}{*}
\newcommand{\dist} {\operatorname{dist}}
\newcommand{\CF}{c_0} 
\newcommand{\kerp}{\mathcal K}

\newcommand{\Cov}{\operatorname{cov}}
\newcommand{\Sym}{\operatorname{Sym}}

\newcommand{\ES}{\mathcal S} 
\newcommand{\EN}{\mathcal N} 
\newcommand{\EM}{\mathcal M} 
\newcommand{\Sc}{\operatorname{Sc}} 
\newcommand{\Ht}{\operatorname{Ht}}

\newcommand{\E}{\operatorname{E}} 
\newcommand{\sign}{\operatorname{sign}} 

\newcommand{\divid}{d} 

\title[Mean Value of $\left|K_{2}(\mathcal{O})\right|$ in Function Fields]{A Simple Proof of the Mean Value of $\left|K_{2}(\mathcal{O})\right|$ in Function Fields \newline \newline \textit{Une démonstration simple de la valeur moyenne de $\left|K_{2}(\mathcal{O})\right|$ en corps de fonctions}} 	
\author{Julio Andrade}
\address{\newline
Mathematical Institute - University of Oxford \newline 
Radcliffe Observatory Quarter \newline
Woodstock Road \newline
Oxford \newline
OX2 6GG \newline
UK} 
\email{buenodeandra@maths.ox.ac.uk}



\begin{abstract}
Let $F$ be a finite field of odd cardinality $q$, $A=F[T]$ the polynomial ring over $F$, $k=F(T)$ the rational function field over $F$ and $\mathcal{H}$ the set of square-free monic polynomials in $A$ of degree odd. If $D\in\mathcal{H}$, we denote by $\mathcal{O}_{D}$ the integral closure of $A$ in $k(\sqrt{D})$. In this note we give a simple proof for the average value of the size of the groups $K_{2}(\mathcal{O}_{D})$ as $D$ varies over the ensemble $\mathcal{H}$ and $q$ is kept fixed. The proof is based on character sums estimates and in the use of the Riemann hypothesis for curves over finite fields.
\newline
\newline
\textbf{Résumé:} \textit{Soit $F$ un corps fini de cardinalité impair $q$, $A=F[T]$ l'anneau de polynomes sur de $F$, $k=F(T)$ Le corps de fonction rationnelle sur $F$ et $\mathcal{H}$ l'ensemble des polynomes unitaires et sans facteur carré en $A$ de degré impair. Si $D\in\mathcal{H}$, on dénoter $\mathcal{O}_{D}$ la clóture intégrale de $A$ en $k(\sqrt{D})$. Dans cette note, nous donnons une preuve simple pour la valeur moyenne de la taille des groupes $K_{2}(\mathcal{O}_{D})$ quand $D$ varie dans l'ensemble $\mathcal{H}$ et $q$ est maintenu fixe. La preuve est basée sur des sommes de caractères et dans l'utilisation de l'hypothèse de Riemann pour les courbes sur les corps finis.}
\end{abstract}
\date{\today}

\maketitle

\section{Introduction}
In \cite{Ros}, the author established average value results for the size of the algebraic $K$ groups $K_{2}(\mathcal{O})$ over function fields. His proof crucially depends on the mean value of quadratic Dirichlet $L$-functions over function fields which in turn was first obtained with the help of functions defined on the metaplectic two-fold cover of $GL(2,k_{\infty})$, where $k_{\infty}$ is the completion of $k=\mathbb{F}_{q}(T)$ at the prime at infinity.

In this paper we provide a simple proof for the mean value of the size of the groups $K_{2}(\mathcal{O})$ over the rational function field $\mathbb{F}_{q}(T)$. Our proof is simpler in the sense that we avoid the Eisenstein series construction involved in the proof given by Hoffstein and Rosen \cite{HR}, and we do this by computing the mean value of the required quadratic Dirichlet $L$-function in $\mathbb{F}_{q}(T)$ through character sum estimates. 

We start by setting the notation. Let $F=\mathbb{F}_{q}$ be a finite field with $q$ elements ($q$ odd), $A=\mathbb{F}_{q}[T]$, and $k=\mathbb{F}_{q}(T)$. For $f\in A$ we define $|f|=q^{\text{deg}(f)}$ if $f\neq0$ and $|0|=0$. If $D\in A$ is square-free then $\mathcal{O}_{D}$ is the integral closure of $A$ in the quadratic function field $K_{D}=k(\sqrt{D})$. 

The zeta function associated to $A$ is defined by

\begin{equation}
\zeta_{A}(s)=\sum_{\substack{f\in A \\ \text{monic}}}|f|^{-s}=\prod_{\substack{P\in A \\ \text{monic} \\ \text{irreducible}}}(1-|P|^{-s})^{-1}.
\end{equation}    
A straightforward calculation shows that $\zeta_{A}(s)=(1-q^{1-s})^{-1}$. If $D\in A$ is square-free we set $\chi_{D}(f)=(D/f)$ where $(D/f)$ is the Kronecker symbol in $A$, and we form the quadratic Dirichlet $L$-function $L(s,\chi_{D})=\sum_{f}\chi_{D}(f)|f|^{-s}$. Lastly, the zeta function of the ring $\mathcal{O}_{D}$ is defined by $\zeta_{\mathcal{O}_{D}}(s)=\sum_{\mathfrak{a}}N\mathfrak{a}^{-s}$ where $\mathfrak{a}$ runs through the nonzero ideals of $\mathcal{O}_{D}=A[\sqrt{D}]$ and $N\mathfrak{a}$ denotes the norm of $\mathfrak{a}$, i.e., the number of elements in $\mathcal{O}_{D}/\mathfrak{a}$. Similar to number fields \cite[Proposition 17.7]{Ros-book}, one has the relation

\begin{equation}
\label{eq:zetaOD}
\zeta_{\mathcal{O}_{D}}(s)=\zeta_{A}(s)L(s,\chi_{D}).
\end{equation}

Making use of \eqref{eq:zetaOD}, together with the results of Quillen \cite{Qui} and Tate \cite{Tate}, Rosen \cite{Ros} was able to relate the number $L(2,\chi_{D})$ to the size of the group $K_{2}(\mathcal{O}_{D})$. We will use such relationship to prove our main result. 

\section{The algebraic $K$ groups $K_{2}(\mathcal{O}_{D})$ and a theorem of Rosen}

Let $D\in A$ be a monic and square-free polynomial. For ease of discussion we only consider the case where degree of $D$ is odd since the case with degree of $D$ even is similar and there are no important differences. Let $F=\mathbb{F}_{q}$ and $K/F$ be a function field in one variable with a finite constant field $\mathbb{F}_{q}$. The primes in $K$ are denoted by $v$, and $\mathcal{O}_{v}$ is the valuation ring at $v$. We denote by $\mathcal{P}_{v}$ the maximal ideal of $\mathcal{O}_{v}$ and by $\overline{F}_{v}$ the residue class field at $v$. The tame symbol $(\ast,\ast)_{v}$ is a mapping from $K^{*}\times K^{*}$ to $\overline{F}_{v}^{*}$ defined by 

\begin{equation}
(a,b)_{v}=(-1)^{v(a)v(b)}a^{v(b)}/b^{v(a)}\ \text{modulo}\ \mathcal{P}_{v}.
\end{equation} 

Let $a\in K^{*}$ such that $a\neq0,1$ so the group $K_{2}(K)$ is defined to be $K^{*}\otimes K^{*}$ modulo the subgroup generated by the elements $a\otimes(1-a)$. Moore (see \cite{Tate} for more details) proved that the following sequence is exact

\begin{equation}
(0)\longrightarrow\ker(\lambda)\longrightarrow K_{2}(K)\overset{\lambda}{\longrightarrow}\bigoplus_{v}\overline{F}_{v}^{*}\overset{\mu}{\longrightarrow}F^{*}\longrightarrow(0),
\end{equation}
where $\lambda:K_{2}(K)\rightarrow\bigoplus_{v}\overline{F}_{v}^{*}$ is the sum of the tame symbol maps, and $\mu:\bigoplus_{v}\overline{F}_{v}^{*}\rightarrow F^{*}$ is the map given by $\mu(\ldots,a_{v},\ldots)=\prod_{v}a_{v}^{m_{v}/m}$ where $m_{v}=N\mathcal{P}_{v}-1$ and $m=\left|F^{*}\right|=q-1$.

By making use of the above discussion with Tate's proof \cite{Tate} of the Birch-Tate conjecture concerning the size of $\ker(\lambda)$, i.e.,

\begin{equation}
|\ker(\lambda)|=(q-1)(q^{2}-1)\zeta_{K}(-1),
\end{equation}
where $\zeta_{K}(s)=\prod_{v}(1-N\mathcal{P}_{v}^{-s})^{-1}$, the product being over all the primes $v$ of the function field $K$, Rosen \cite[Proposition 2]{Ros} established that

\begin{equation}
\label{eq:numberK2}
\#K_{2}(\mathcal{O}_{D})=q^{(3/2)\text{deg}(D)}q^{-3/2}L(2,\chi_{D}).
\end{equation}

With this in hand Rosen \cite[Proposition 2(a)]{Ros} proves the following

\begin{theorem}[Rosen]
\label{thm:Rosen}
Let $m$ be a square-free polynomial of degree $M$, with $M$ odd, and $\varepsilon>0$ given. Then

\begin{equation}
\label{eq:Rosen}
(q-1)^{-1}(q^{M}-q^{M-1})^{-1}\sum_{\substack{m\in A \\ m \ \mathrm{square-free}}}|K_{2}(\mathcal{O}_{m})|=\zeta_{A}(2)\zeta_{A}(4)c(2)q^{-3/2}q^{3(M/2)}+O(q^{M(1+\varepsilon)}),
\end{equation} 
where
\begin{equation}
c(2)=\prod_{P}\left(1-|P|^{-2}-|P|^{-5}+|P|^{-6}\right),
\end{equation}
the product is taken over all monic irreducible polynomials in $A$.
\end{theorem}

\section{The Main Result}

Without further postponements we present below the main result of this note.

\begin{theorem}
\label{mainthmnote}
Let $\mathcal{H}=\left\{\text{$D\in A$ {\rm monic, square-free and deg$(D)=2g+1$}}\right\}$ and $\varepsilon>0$. Then

\begin{equation}
\label{eq:mainthm}
\frac{1}{\#\mathcal{H}}\sum_{D\in\mathcal{H}}\#K_{2}(\mathcal{O}_{D})=q^{\tfrac{3}{2}(2g+1)}q^{-3/2}\zeta_{A}(4)P(4)+O(q^{(2g+1)(1+\varepsilon)}),
\end{equation}
where

\begin{equation}
P(s)=\prod_{\substack{P\in A \\ {\rm monic} \\ {\rm irreducible}}}\left(1-\frac{1}{(|P|+1)|P|^{s}}\right).
\end{equation}
\end{theorem}

A theorem similar to this was previously studied by Rosen. There are essentially two differences between Rosen's result (Theorem \ref{thm:Rosen}) and the main result in this paper. First is that the average value of $|K_{2}(\mathcal{O}_{D})|$, as presented by Rosen, is an average taken over \textit{all} square-free $D$ and in our result we only consider the \textit{monic} and square-free $D$, i.e., positive and fundamental discriminants over function fields. Comparing equations \eqref{eq:Rosen} and \eqref{eq:mainthm} we observe that the constants multiplying the main term are close but not equal. And this is due to the fact that in our result we are summing over monic and square-free while in Rosen's result he is summing over all square-free, and in this sense Rosen's result is more general. This phenomenon is not new and it has appeared when you compare the main theorem of \cite{And-Kea} with \cite[Theorem 5.2]{HR} and it also appears in number fields when you compare the first moment of quadratic Dirichlet $L$-functions at the central point as given by \cite[Theorem 1]{Jutila} and \cite[Theorem (1)]{GH}, in both comparisons we see again that the constants multiplying the leading term in the mean values are different. The second, and most important difference, is the argument used to prove such result. In Rosen, he needs to invoke a mean value of $L(s,\chi_{D})$ that was previously proved by himself and Hoffstein \cite{HR} through the use of the theory of Eisenstein series and the metaplectic two-fold cover of $GL(2,k_{\infty})$, whereas our method is solely based on estimating characters sums and in the use of the Riemann hypothesis for curves over finite fields.

\section{Preparatory Results}

In this section we present a few auxiliary results that will be used in the proof of the main theorem of this note.

\begin{lemma}
\label{lemma1}
Let $f\in A$ be a fixed monic polynomial. Then for all $\varepsilon>0$ we have that
\begin{equation}
\sum_{\substack{D\in\mathcal{H} \\ \mathrm{gcd}(D,f)=1}}1=\frac{|D|}{\zeta_{A}(2)}\prod_{\substack{P \ \mathrm{monic} \\ \mathrm{irrducible} \\ P\mid f}}\left(\frac{|P|}{|P|+1}\right)+O\left(|D|^{\tfrac{1}{2}}|f|^{\varepsilon}\right).
\end{equation}
\end{lemma}

\begin{proof}
See \cite[Proposition 5.2]{And-Kea}.
\end{proof}

We also need the following lemma.

\begin{lemma}
\label{lemma2}
We have 
\begin{equation}
\sum_{\substack{f \ \mathrm{monic} \\ \mathrm{deg}(f)=n}}\prod_{P\mid f}(1+|P|^{-1})^{-1}=q^{n}\sum_{\substack{d \ \mathrm{monic} \\ \mathrm{deg}(d)\leq n}}\frac{\mu(d)}{|d|}\prod_{P\mid d}(|P|+1)^{-1}.
\end{equation}
\end{lemma}

\begin{proof}
See \cite[Lemma 5.7]{And-Kea}.
\end{proof}

The last result that we need before to proceed to the proof of our main theorem is given below and it has appeared in a different form in \cite{And-Kea,And-Kea2,Fai-Rud} and its proof, as appears here, was first given in \cite{Andrade1}.

\begin{lemma}
\label{lemma3}
If $f\in A$ is not a perfect square then

\begin{equation}
\label{4.8}
\sum_{\substack{D\in\mathcal{H} \\ f\neq\square}}\left(\frac{D}{f}\right)\ll|D|^{1/2}|f|^{1/4}.
\end{equation}
\end{lemma}

\begin{proof}
we write

\begin{align}
\label{4.11}
&\sum_{\substack{D\in\mathcal{H}}}\left(\frac{D}{f}\right)=\sum_{2\alpha+\beta=2g+1}\sum_{\text{deg}(B)=\beta}\sum_{\text{deg}(A)=\alpha}\mu(A)\left(\frac{A^{2}B}{f}\right)\nonumber\\
&=\sum_{0\leq\alpha\leq g}\sum_{\text{deg}(A)=\alpha}\mu(A)\left(\frac{A^{2}}{f}\right)\sum_{\text{deg}(B)=2g+1-2\alpha}\left(\frac{B}{f}\right)\nonumber\\
&\leq\sum_{0\leq\alpha\leq g}\sum_{\text{deg}(A)=\alpha}\sum_{\text{deg}(B)=2g+1-2\alpha}\left(\frac{B}{f}\right).
\end{align}
If $f\neq\square$ then $\sum_{\text{deg}(B)=2g+1-2\alpha}\left(\frac{B}{f}\right)$ is a character sum to a non-principal character modulo $f$. So using \cite[Proposition 2.1]{Hsu} (which is the P\'{o}lya-Vinogradov inequality for $\mathbb{F}_{q}[T]$) we have that

\begin{equation}
\label{4.12}
\sum_{\text{deg}(B)=2g+1-2\alpha}\left(\frac{B}{f}\right)\ll|f|^{1/2}.
\end{equation}
Further we can estimate trivially the non-principal character sum by

\begin{equation}
\label{4.13}
\sum_{\text{deg}(B)=2g+1-2\alpha}\left(\frac{B}{f}\right)\ll\frac{|D|}{|A|^{2}}=q^{2g+1-2\alpha}.
\end{equation}
Thus, if $f\neq\square$, we obtain that 

\begin{align}
\label{4.14}
\sum_{D\in\mathcal{H}}\left(\frac{D}{f}\right)&\ll\sum_{0\leq\alpha\leq g}\sum_{\text{deg}(A)=\alpha}\text{min}\left(|f|^{1/2},\frac{|D|}{|A|^{2}}\right)\nonumber \\
&\ll|D|^{\tfrac{1}{2}}|f|^{\tfrac{1}{4}},
\end{align}
upon using the first bound \eqref{4.12} for $\alpha\leq g-\frac{\text{deg}(f)}{4}$ and the second bound \eqref{4.13} for larger $\alpha$. And this concludes the proof of the lemma.
\end{proof}

\section{Proof of the Main Theorem}
From now on we are assuming that all the sums are being taken over monic polynomials and the products are over monic and irreducible polynomials $P$ in $\mathbb{F}_{q}[T]$.

By \cite[Proposition 4.3]{Ros-book} we have
\begin{eqnarray}
\label{eq:split}
\sum_{D\in\mathcal{H}}L(2,\chi_{D})&=&\sum_{D\in\mathcal{H}}\sum_{\text{deg}(f)\leq 2g}\chi_{D}(f)|f|^{-2} \nonumber\\
&=&\sum_{D\in\mathcal{H}}\sum_{\substack{\text{deg}(f)\leq 2g \\ f=\square}}\chi_{D}(f)|f|^{-2}+\sum_{D\in\mathcal{H}}\sum_{\substack{\text{deg}(f)\leq 2g \\ f\neq\square}}\chi_{D}(f)|f|^{-2}.
\end{eqnarray}
For the sum above, where $f$ is not a square of a polynomial, we use Lemma \ref{lemma3}, which depends on the Riemann hypothesis for curves over finite fields, to obtain that

\begin{equation}
\label{eq:nonsquare}
\sum_{D\in\mathcal{H}}\sum_{\substack{\text{deg}(f)\leq 2g \\ f\neq\square}}\chi_{D}(f)|f|^{-2}\ll q^{g}.
\end{equation}
For the sum with $f$ a square of a polynomial in \eqref{eq:split} we need some extra manipulations. First we use Lemma \ref{lemma1} so we can write

\begin{align}
\sum_{D\in\mathcal{H}}\sum_{\substack{\text{deg}(f)\leq 2g \\ f=\square}}\chi_{D}(f)|f|^{-2}=\frac{|D|}{\zeta_{A}(2)}\sum_{m=0}^{g}\frac{1}{q^{4m}}\sum_{\text{deg}(l)=m}\prod_{P\mid l}\left(\frac{|P|}{|P|+1}\right)+O\left(|D|^{1/2}\sum_{n=0}^{2g}q^{n\epsilon-n}\right).
\end{align}
From Lemma \ref{lemma2} we have

\begin{eqnarray}
\label{eq:maintermalmost}
\sum_{D\in\mathcal{H}}\sum_{\substack{\text{deg}(f)\leq 2g \\ f=\square}}\chi_{D}(f)|f|^{-2}&=&\frac{|D|}{\zeta_{A}(2)}\sum_{\text{deg}(d)\leq g}\frac{\mu(d)}{|d|}\prod_{P\mid d}\frac{1}{|P|+1}\sum_{\text{deg}(d)\leq m\leq g}q^{-3m} \nonumber \\
&+&O\left(q^{-g}\frac{(q^{2g\epsilon+\epsilon}-q^{2g+1})}{q^{\epsilon}-q}\right).
\end{eqnarray}

After some arithmetic manipulations and summing the geometric series we can rewrite \eqref{eq:maintermalmost} as

\begin{eqnarray}
\label{eq:sums over d}
\sum_{D\in\mathcal{H}}\sum_{\substack{\text{deg}(f)\leq 2g \\ f=\square}}\chi_{D}(f)|f|^{-2}&=&\zeta_{A}(4)\frac{|D|}{\zeta_{A}(2)}\left\{\sum_{d \ \text{monic}}-\sum_{\text{deg}(d)>g}\right\}\left(\frac{\mu(d)}{|d|^{4}}\prod_{P\mid d}\frac{1}{|P|+1}\right)\nonumber\\
&-&\frac{q^{-3g}}{q^{3}-1}\frac{|D|}{\zeta_{A}(2)}\left\{\sum_{d \ \text{monic}}-\sum_{\text{deg}(d)>g}\right\}\left(\frac{\mu(d)}{|d|}\prod_{P\mid d}\frac{1}{|P|+1}\right)\nonumber \\
&+&O\left(q^{-g}\frac{(q^{2g\epsilon+\epsilon}-q^{2g+1})}{q^{\epsilon}-q}\right).
\end{eqnarray}

The sums over $\text{deg}(d)>g$ in \eqref{eq:sums over d} are respectively bounded by $O(q^{-4g})$ and $O(q^{-g})$ as can be seen from below,

\begin{eqnarray}
\sum_{\text{deg}(d)>g}\frac{\mu(d)}{|d|^{4}}\prod_{P\mid d}(|P|+1)^{-1}&\ll&\sum_{\text{deg}(d)>g}\frac{1}{|d|^{4}}\prod_{P\mid d}\frac{1}{|P|}\ll\sum_{\text{deg}(d)>g}\frac{1}{|d|^{5}}\nonumber \\
&=& \sum_{n>g}q^{-4n}\ll q^{-4g}
\end{eqnarray}
and

\begin{eqnarray}
\sum_{\text{deg}(d)>g}\frac{\mu(d)}{|d|}\prod_{P\mid d}(|P|+1)^{-1}&\ll&\sum_{\text{deg}(d)>g}\frac{1}{|d|}\prod_{P\mid d}\frac{1}{|P|}\ll\sum_{\text{deg}(d)>g}\frac{1}{|d|^{2}}\nonumber \\
&=& \sum_{n>g}q^{-n}\ll q^{-g},
\end{eqnarray}
and therefore does not contribute to the main term. 

By expressing the sums over all monic $d$ in \eqref{eq:sums over d} as Euler products we derive that

\begin{equation}
\label{eq:square contribution}
\sum_{D\in\mathcal{H}}\sum_{\substack{\text{deg}(f)\leq 2g \\ f=\square}}\chi_{D}(f)|f|^{-2}=\zeta_{A}(4)\frac{|D|}{\zeta_{A}(2)}P(4)+O\left(q^{-g}\frac{(q^{2g\epsilon+\epsilon}-q^{2g+1})}{q^{\epsilon}-q}\right),
\end{equation}
where $P(s)$ is given as in the statement of Theorem \ref{mainthmnote}.

Combining \eqref{eq:nonsquare} and \eqref{eq:square contribution} we get that

\begin{equation}
\label{eq:meanvalueL}
\sum_{D\in\mathcal{H}}L(2,\chi_{D})=\frac{|D|}{\zeta_{A}(2)}\zeta_{A}(4)P(4)+O(q^{g})+O\left(\frac{q^{g}+q^{g(2\epsilon-1)+\epsilon}}{q^{\epsilon}-q}\right).
\end{equation}

We invoke \cite[Proposition 2.3]{Ros-book}, which shows that $\#\mathcal{H}=|D|/\zeta_{A}(2)$, together with equation \eqref{eq:numberK2} and a few arithmetic maneuvers to complete the proof of the main theorem in this letter. \qed

\vspace{0.5cm}

\noindent \textbf{Acknowledgment.} This research was supported by EPSRC grant EP/K021132X/1. 

\noindent The author is thankful to the comments of an anonymous referee which helped to give more clarity to the presentation of this note. The author also wishes to thank Professor Alain Connes for the several discussions related to the problem treated in this paper.




\end{document}